\documentclass[a4paper,10pt]{amsart}
\usepackage{amsmath,amsfonts,amssymb,amsthm, amsfonts}
\usepackage[latin1]{inputenc}
\usepackage[english]{babel}

\usepackage[cmtip,arrow]{xy}
\usepackage{pb-diagram,pb-xy}

\swapnumbers 
\theoremstyle{plain}
\newtheorem{thm}{Theorem}[section]
\newtheorem{prop}[thm]{Proposition}
\newtheorem{lemma}[thm]{Lemma}
\newtheorem{cor}[thm]{Corollary}
\theoremstyle{remark}
\theoremstyle{definition}
\newtheorem{rem}[thm]{Remark}
\newtheorem{rems}[thm]{Remarks}
\newtheorem{remdef}[thm]{Remark-Definition}
\newtheorem{remsdefs}[thm]{Remarks-Definitions}

\newtheorem{defi}[thm]{Definition}
\newtheorem{defis}[thm]{Definitions}

\newtheorem{recalls}[thm]{Recalls}
\newtheorem{notas}[thm]{Notations}
\newtheorem{examples}[thm]{Examples}

\title[Some additive decompositions]{Some additive decompositions \\
of $\mathbb{K}$-regular matrices}

\author[A. Dolcetti]{Alberto Dolcetti}
\author[D. Pertici]{Donato Pertici}
\address{Dipartimento di Matematica  e Informatica ``Ulisse Dini''\\Viale Morgagni 67/a\\50134 Firenze, ITALIA}
\email{alberto.dolcetti@unifi.it, \ donato.pertici@unifi.it}

\begin{document}
\parindent 0pt

\selectlanguage{english}

\begin{abstract}
Under suitable hypotheses on the ground field and on the matrix $M$, we discuss existence, uniqueness and properties of some additive decompositions of $M$ and of its image through a convergent series. 
\end{abstract}

\maketitle

\tableofcontents

\renewcommand{\thefootnote}{\fnsymbol{footnote}}
\footnotetext{
This research was partially supported by MIUR-PRIN: ``Variet\`a reali e complesse: geo\-me\-tria, topologia e analisi armonica'' and by GNSAGA-INdAM.}
\renewcommand{\thefootnote}{\arabic{footnote}}
\setcounter{footnote}{0}


\medskip

{\scshape{Mathematics~Subject~Classification~(2010):}}  15A21, 12F10, 12J10, 12J15, 15A18.

\section*{Introduction}

We work over a fixed field $\mathbb{K}$ and choose a fixed algebraic closure, denoted by $\overline{\mathbb{K}}$.

In \S 1 we consider the set $\mathcal{R}_\mathbb{K}$ of the $\mathbb{K}$-\emph{regular elements}, i.e. the set of the elements of $\overline{\mathbb{K}}$ whose degree on ${\mathbb{K}}$ is not a multiple of the characteristic of $\mathbb{K}$ (Definition \ref{def-K-reg}). If $char(\mathbb{K})=0$, then $\mathcal{R}_\mathbb{K}= \overline{\mathbb{K}}$, while in general it is a subset of separable elements over $\mathbb{K}$ (Remarks \ref{prime-propr-C_K}). 

On $\mathcal{R}_\mathbb{K}$, we can define a $\mathbb{K}$-\emph{projection map} $H_\mathbb{K}$  having image into $\mathbb{K}$ and the related zero-locus $Ker(H_\mathbb{K}) =
\{ \beta \in \mathcal{R}_\mathbb{K} \ / \ H_\mathbb{K}(\beta)=0 \}$ (Definitions \ref{K-proj}).

Analogously to the real-complex case, this map allows to decompose, in a natural and unique way, every element $\lambda \in \mathcal{R}_\mathbb{K}$ as a sum of an element $H_\mathbb{K}(\lambda)$ of $\mathbb{K}$ and of an element $V_\mathbb{K}(\lambda) = \lambda -H_\mathbb{K}(\lambda)$
of $Ker(H_\mathbb{K})$ (Theorem \ref{decomp-traccia}). This decomposition is called $\mathbb{K}$-\emph{decomposition of} $\lambda$ in its $\mathbb{K}$-\emph{horizontal} and $\mathbb{K}$-\emph{vertical components} (Definition \ref{def-hor-ver}).

We pay a particular attention to case of elements of degree $2$ over the field $\mathbb{K}$, supposed to have characteristic different from $2$. In this case the unique conjugate over $\mathbb{K}$ of such an element is the image of its $\mathbb{K}$-\emph{involution} (Remark \ref{involuzione} and Proposition \ref{car-non-2}).

\smallskip

A \emph{complete additive Jordan-Chevalley decomposition over} $\mathbb{K}$ of $M \in M_n(\mathbb{K})$ (Definition \ref{Def-CAJC-Dec}) is an additive decomposition $M = H + V + N$, 
where $H, V, N \in M_n(\mathbb{K})$ are mutually commuting, $H$ is \emph{diagonalizable} over $\mathbb{K}$, $V$ is $\mathbb{K}$-\emph{vertical} and semisimple (i.e. it is semisimple and its eigenvalues are in $Ker(H_\mathbb{K})$) and $N$ is \emph{nilpotent}. 

In \S 2 we prove that a matrix $M \in M_n(\mathbb{K})$ has a uniquely determined complete additive Jordan-Chevalley decomposition over $\mathbb{K}$ provided that $M$ is $\mathbb{K}$-\emph{regular}  (Theorem \ref{Delta-Sigma-N}), i.e. 
all its eigenvalues are in $\mathcal{R}_\mathbb{K}$ or equivalently no irreducible component over $\mathbb{K}$ of its minimal polynomial has degree multiple of the characteristic of $\mathbb{K}$ (Definition \ref{def-matr-CR}).

\smallskip

From \S 3 we focus on nonzero matrices $M \in M_n(\mathbb{K})$ which are semisimple with \emph{splitting bound over} $\mathbb{K}$ at most $2$, i.e every irreducible component over $\mathbb{K}$ of its minimal polynomial of $M$ has degree at most $2$ (Definitions \ref{splitting-bound}). In particular, if $char(\mathbb{K}) \ne 2$, these matrices are $\mathbb{K}$-regular.

Moreover these matrices can be characterized from having a unique \emph{fine Frobenius decomposition} written in terms of certain matrices, called \emph{fine Frobenius covariants} of $M$, which are polynomial functions of $M$ and in terms of the eigenvectors of $M$ together with their $\mathbb{K}$-decompositions (Definition \ref{def-fine-Frob-Dec} and Proposition \ref{dec-fine-unica}). The fine Frobenius decomposition is a key tool for the remaining part of this note and improves the analysis of the \emph{Frobenius decomposition} in \cite{DoPe2017} \S 1.

\smallskip

In \S 4 we assume that the ground field $\mathbb{K}$ has an absolute value and $char(\mathbb{K}) \ne 2$. We consider the image of nonzero matrices in $\mathbb{K}$, which are semisimple with splitting bound at most $2$ over $\mathbb{K}$ throughout a convergent series $f(X)$. The fine Frobenius decomposition of a matrix allows to write its image in terms of the images throughout $f$ of its eigenvalues and of its fine Frobenius covariants (Proposition \ref{decf(M)} and Examples \ref{exp-cos'}).

\smallskip

In \S 5 we collect some properties of \emph{ordered quadratically closed fields} and of \emph{real closed fields} (Definition \ref{defQCO} and Definition \ref{defRC}); in both cases the characteristic is $0$. Similarly to the ordinary real case, every element in the algebraic closure of these fields has a decomposition in a \emph{real part} and in an \emph{imaginary part} (Proposition \ref{parti-Re-Im}). We end the section with some properties of the completion of a real closed valued field (Proposition \ref{completamento-RCVal} and Corollary \ref{cor-dopo-completamento}).

\smallskip

Finally in \S 6 we consider matrices with coefficients in $\mathbb{K}$ supposed to be ordered quadratically closed or supposed to be real closed. This allows to simplify and to specify the fine Frobenius decomposition of a matrix $M$ by a suitable \emph{normalization} of its Frobenius covariants (Remark-Definition \ref{norm-Frob}). When the ground field has an absolute value too, the image of $M$ throughout a convergent series is consequently simplified and specified (Proposition \ref{**}, Remark \ref{caso-reale}). 
As an example we extend the classical Rodrigues' formula, formerly given for the exponential of a real skew symmetric matrix (Examples \ref{exp-cos} and Remark \ref{caso-reale}).
By the way we give some further simple characterizations of real closed fields in terms of $Ker(H_\mathbb{K})$, of the $\mathbb{K}$-involution and of the splitting bound of matrices (Proposition \ref{real-closed-equiv}).

\medskip

\textbf{Acknowledgement.} We want to thank Virgilio Pannone and Orazio Puglisi for many discussions and suggestions about the matter of \S 1 of this paper.

\section{$\mathbb{K}$-regular elements}

\begin{defi}\label{def-K-reg}
We say that an element of $\overline{\mathbb{K}}$ is $\mathbb{K}$-\emph{regular}, if its degree over ${\mathbb{K}}$ is not a multiple of  $char(\mathbb{K})$: the \emph{characteristic} of $\mathbb{K}$. 

We denote by $\mathcal{R}_\mathbb{K}$ the set of these elements.
\end{defi}

\begin{rems}\label{prime-propr-C_K}

a) If $char(\mathbb{K}) = 0$, then $\mathcal{R}_\mathbb{K} = \overline{\mathbb{K}}$.

\smallskip

b) In general $\mathbb{K} \subseteq \mathcal{R}_\mathbb{K}  \subseteq \mathbb{K}^{sep} \subseteq \overline{\mathbb{K}}$, where the $\mathbb{K}^{sep}$ is the \emph{separable closure} of $\mathbb{K}$ in $\overline{\mathbb{K}}$.

In positive characteristic the inclusion $\mathcal{R}_\mathbb{K}  \subseteq \mathbb{K}^{sep}$ follows from the fact that irreducible polynomials in $\mathbb{K}[X]$ of degree not multiple of $char(\mathbb{K})$ are separable (see for instance \cite{Hun1974} V,\,Cor.\,6.14).

\smallskip

c) Note that if $\lambda \in \mathcal{R}_\mathbb{K}$ and $\lambda' \in \mathbb{K}(\lambda)$, then $\lambda' \in \mathcal{R}_\mathbb{K}$, being its degree over $\mathbb{K}$ a divisor of the degree of $\lambda$ over $\mathbb{K}$. Hence: $\mathbb{K}(\lambda) \subseteq \mathcal{R}_\mathbb{K}$ and so 
\begin{center}
$\mathcal{R}_\mathbb{K} = \bigcup_{\lambda \in \mathcal{R}_\mathbb{K}} \mathbb{K}(\lambda).$
\end{center}

d) If $char(\mathbb{K}) = p > 0$, then the inclusion $\mathcal{R}_\mathbb{K} \subseteq \mathbb{K}^{sep}$ is generally strict. 

Indeed, if $a \in \mathbb{K}$, then either the polynomial $f_a(X) = X^p -X -a$ is irreducible and separable over $\mathbb{K}$ or it has $p$ distinct roots in $\mathbb{K}$ (see for instance \cite{Lang2002} VI,\,6,\, Thm.\,6.4). 

When $\mathbb{K} = \mathbb{F}_{p^m}$ (the finite field with $p^m$ elements), by reasons of cardinality, there exists an element $a \in \mathbb{F}_{p^m}$ such that $f_a(X)$ is irreducible and separable over $\mathbb{F}_{p^m}$. So its roots are in  $\mathbb{F}_{p^m}^{sep}$, but not in $\mathcal{R}_{\mathbb{F}_{p^m}}$.

\smallskip

e) Assume that $char(\mathbb{K}) = p > 0$ and let $\lambda_1, \cdots \lambda_l \in \mathcal{R}_\mathbb{K}$ with degrees over $\mathbb{K}$ respectively $d_1, \cdots,  d_l$ such that $\dfrac{d_1 \cdots d_l}{lcm(d_1, \cdots, d_l)} < p$. 
Then $\mathbb{K}(\lambda_1, \cdots \lambda_l) \subseteq \mathcal{R}_\mathbb{K}$.

Indeed $[\mathbb{K}(\lambda_1, \cdots \lambda_l) : \mathbb{K}] \le d_1 \cdots d_l$ and every $d_i$  divides $[\mathbb{K}(\lambda_1, \cdots \lambda_l) : \mathbb{K}]$.  
So $lcm(d_1, \cdots, d_l)$ divides $[\mathbb{K}(\lambda_1, \cdots \lambda_l) : \mathbb{K}]$, therefore 

$[\mathbb{K}(\lambda_1, \cdots \lambda_l) : \mathbb{K}] = lcm(d_1, \cdots, d_l) \, h $ where $h$ is a positive integer at most equal to $\dfrac{d_1 \cdots d_l}{lcm(d_1, \cdots, d_l)}$. Note that $p$ does not divides any $d_i$'s, so it does not divide $lcm(d_1, \cdots, d_l)$. From $p >\dfrac{d_1 \cdots d_l}{lcm(d_1, \cdots, d_l)} \ge h$, we get that $p$ does not divide $[\mathbb{K}(\lambda_1, \cdots \lambda_l) : \mathbb{K}]$. 
This suffices to conclude that $\mathbb{K}(\lambda_1, \cdots \lambda_l) \subseteq \mathcal{R}_\mathbb{K}$.

Note that in case of $l=2$, we have $\dfrac{d_1 \, d_2}{lcm(d_1, d_2)} = gcd(d_1, d_2)$.

\smallskip

f) Recall that two elements of $\lambda, \lambda' \in \overline{\mathbb{K}}$ are said to be \emph{conjugated over} $\mathbb{K}$, if they have the same (monic) minimal polynomial over $\mathbb{K}$. 

It is trivial to get that $\lambda \in \mathcal{R}_\mathbb{K}$ if and only if $\lambda' \in \mathcal{R}_\mathbb{K}$.
\end{rems}

\begin{rem}\label{traccia}
Let $\mathbb{L}$ be any finite extension of $\mathbb{K}$, contained in $\overline{\mathbb{K}}$ and 
$\lambda \in \mathbb{L}$ be of degree $d$ over $\mathbb{K}$ with minimal polynomial over $\mathbb{K}$: 

$f(X) = X^d + a_{d-1}X^{d-1} + \cdots + a_1X + a_0$.

The \emph{trace of} $\lambda$ \emph{relative to the extension} $\mathbb{L}/ \mathbb{K}$,  is the $\mathbb{K}$-linear mapping from $\mathbb{L}$ into $\mathbb{K}$ defined by
\begin{center}
$T_{\mathbb{L}/ \mathbb{K}}(\lambda) = -[\mathbb{L} : \mathbb{K}(\lambda)] a_{d-1}$
\end{center}
(see for instance \cite{Lang2002} VI.5 and \cite{Hun1974} V,\,Thm.\,7.3).
\end{rem}

\begin{defis}\label{K-proj}
Let $\lambda \in \mathcal{R}_\mathbb{K}$ and $f(X) = X^d + a_{d-1}X^{d-1} + \cdots + a_1X + a_0$ be its minimal polynomial over $\mathbb{K}$ (so it is possible to divide by $d$ in $\mathbb{K}$).

\smallskip

a) We define the $\mathbb{K}$-\emph{projection} $H_\mathbb{K}:  \mathcal{R}_\mathbb{K} \to \mathbb{K}$ by 
\begin{center}
$H_\mathbb{K}(\lambda)= \dfrac{T_{\mathbb{K}(\lambda) / \mathbb{K}}(\lambda)}{[\mathbb{K}(\lambda) : \mathbb{K}]} = - \dfrac{a_{d-1}}{d}.$
\end{center}

\smallskip

b) We set 
$Ker(H_\mathbb{K})  =
\{ \lambda \in \mathcal{R}_\mathbb{K} \ / \ H_\mathbb{K}(\lambda)=0 \}$.

\smallskip

c) Finally we call $\mathbb{K}$\emph{-norm of} $\lambda$ the element of $\mathbb{K}$ given by
$
N_{\mathbb{K}}(\lambda) = (-1)^d a_0.
$
\end{defis}

\begin{rems}\label{propr-D-K}
a) From the definition we get that the restriction of $H_\mathbb{K}$ on $\mathbb{K}$ is the identity map on $\mathbb{K}$, so $Ker(H_{\mathbb{K}}) \cap \mathbb{K}= \{0\}$.

\smallskip

b) Elements of $\mathcal{R}_\mathbb{K}$, conjugated over $\mathbb{K}$, have the same $\mathbb{K}$-projection and the same $\mathbb{K}$-norm. 
Therefore, if $Ker(H_\mathbb{K})$ contains an element $\lambda$, then it contains the whole conjugacy class of $\lambda$ over $\mathbb{K}$.

\smallskip

c) If $\lambda \in \mathcal{R}_{\mathbb{K}}$ and $\mathbb{L}$ is a finite extension of $\mathbb{K}$ containing $\lambda$ and such that $[\mathbb{L} : \mathbb{K}]$ is not divisible by $char(\mathbb{K})$, then $H_\mathbb{K}(\lambda) = \dfrac{T_{\mathbb{L} / \mathbb{K}} (\lambda)}{[\mathbb{L} : \mathbb{K}]}$.

Indeed, remembering \ref{traccia} and in \ref{K-proj}, we get:
\begin{center}
$H_\mathbb{K}(\lambda) = - \dfrac{a_{d-1}}{d}= \dfrac{T_{\mathbb{L}/\mathbb{K}}(\lambda)}{[\mathbb{L}:\mathbb{K}(\lambda)]\cdot[\mathbb{K}(\lambda): \mathbb{K}]} = \dfrac{T_{\mathbb{L}/\mathbb{K}}(\lambda)}{[\mathbb{L}:\mathbb{K}]}$.
\end{center}

\smallskip

d) If $\mathbb{M}$ is an extension of $\mathbb{K}$, contained in $\mathcal{R}_\mathbb{K}$, then the restriction 
$H_{\mathbb{K}}\vert_{\mathbb{M}}: \mathbb{M} \to {\mathbb{K}}$ is ${\mathbb{K}}$-linear.

In particular if $\mathcal{R}_{\mathbb{K}} = \overline{\mathbb{K}}$ (for instance if $char(\mathbb{K})=0$), then $Ker(H_{\mathbb{K}})$ is a $\mathbb{K}$-vector subspace of $\overline{\mathbb{K}}$.

For, let $\lambda_1, \lambda_2 \in \mathbb{M}$ and $k_1, k_2 \in \mathbb{K}$. Then $\mathbb{L}= \mathbb{K}(\lambda_1, \lambda_2) \subseteq \mathbb{M}\subseteq \mathcal{R}_\mathbb{K} \subseteq \mathbb{K}^{sep}$, therefore $\mathbb{L}$ is a finite separable extension of $\mathbb{K}$.
By primitive element theorem, there exists $\gamma \in \mathbb{L} \subseteq \mathcal{R}_\mathbb{K}$, such that $\mathbb{L} = \mathbb{K}(\gamma)$. Since $\gamma \in \mathcal{R}_\mathbb{K}$, $[\mathbb{L}: \mathbb{K}]$ is not divisible by $p$.
Hence by the remark (c) above, the restriction of $H_{\mathbb{K}}$ on ${\mathbb{L}}$ is $\mathbb{K}$-linear and so $H_{\mathbb{K}}(k_1\lambda_1 + k_2 \lambda_2) = k_1H_{\mathbb{K}}(\lambda_1) + k_2H_{\mathbb{K}}(\lambda_2)$. Hence $H_{\mathbb{K}}$ is $\mathbb{K}$-linear on $\mathbb{M}$.

\smallskip

e) Assume that $char(\mathbb{K}) = p > 0$ and let $\lambda_1, \cdots , \lambda_l \in Ker(H_\mathbb{K})$ with degrees over $\mathbb{K}$ respectively $d_1, \cdots, d_l$ such that $\dfrac{d_1 \cdots d_l}{lcm(d_1, \cdots, d_l)} < p$. 

Then the $\mathbb{K}$-vector space generated by $\lambda_1, \cdots , \lambda_l$ is contained in $Ker(H_\mathbb{K})$.

Indeed by \ref{prime-propr-C_K} (e), $\mathbb{K} \subseteq \mathbb{K}(\lambda_1, \cdots , \lambda_l) \subseteq \mathcal{R}_\mathbb{K}$ and, by (d) above, $H_{\mathbb{K}}$ is $\mathbb{K}$-linear on  $\mathbb{K}(\lambda_1, \cdots , \lambda_l)$, therefore every $\mathbb{K}$-linear combination of $\lambda_1, \cdots , \lambda_l$ belongs to $Ker(H_\mathbb{K})$.
\end{rems}

\begin{thm}\label{decomp-traccia}
Let $\lambda \in \mathcal{R}_\mathbb{K}$ and set $V_\mathbb{K}(\lambda) = \lambda -H_\mathbb{K}(\lambda)$. 

Then  $\lambda = H_\mathbb{K}(\lambda) + V_\mathbb{K}(\lambda)$ is the unique way to write $\lambda$ as sum of an element $H_\mathbb{K}(\lambda)$ of $\mathbb{K}$ and of an element $V_\mathbb{K}(\lambda)$ of $Ker(H_\mathbb{K})$.
\end{thm}

\begin{proof}
Indeed we can write $\lambda = \alpha + \beta$ with $\alpha = H_\mathbb{K}(\lambda) \in \mathbb{K}$ and $ \beta = V_\mathbb{K}(\lambda) = \lambda -H_\mathbb{K}(\lambda)$, hence $\mathbb{K}(\beta)= \mathbb{K}(\lambda)= \mathbb{L}$ and so $\lambda$ and $\beta$ have the same degree, say $d$, over $\mathbb{K}$, hence $\beta \in \mathcal{R}_\mathbb{K}$. Now $T_{\mathbb{L} / {\mathbb{K}}} (\beta) =  T_{\mathbb{L} / {\mathbb{K}}} (\lambda) -  T_{\mathbb{L} / {\mathbb{K}}} (\alpha) = d \alpha - [\mathbb{L}: \mathbb{K}] \, \alpha =0$. Hence, by \ref{propr-D-K} (c), $\beta \in Ker(H_{\mathbb{K}})$.

For the uniqueness, let $\lambda = \alpha' + \beta'$ be another decomposition of the same type. As above $\beta' \in \mathbb{K}(\lambda) = \mathbb{L}$. By  \ref{propr-D-K} (c), 
$\alpha= H_{\mathbb{K}}(\lambda) = H_{\mathbb{K}}(\alpha')  + H_{\mathbb{K}}(\beta') =  \alpha'$. From this $\beta = \beta'$ too and so we can conclude.
\end{proof}

\begin{defi}\label{def-hor-ver}
We refer to the decomposition in \ref{decomp-traccia} as \emph{the} $\mathbb{K}$-\emph{decomposition of} $\lambda \in \mathcal{R}_\mathbb{K}$,  to $H_\mathbb{K}(\lambda) \in \mathbb{K}$ as the $\mathbb{K}$\emph{-horizontal} \emph{component} of $\lambda$ and to $V_\mathbb{K}(\lambda)\in Ker(H_\mathbb{K})$ as the $\mathbb{K}$\emph{-vertical} \emph{component} of $\lambda$.
\end{defi}

\begin{rem}\label{car0sommadiretta}
If $\mathcal{R}_\mathbb{K} = \overline{\mathbb{K}}$, then $\overline{\mathbb{K}}$ is direct sum of the $\mathbb{K}$-subspaces $\mathbb{K}$ and $Ker(H_\mathbb{K})$: $\overline{\mathbb{K}} = \mathbb{K} \oplus	Ker(H_\mathbb{K})$.
\end{rem}

\begin{prop}\label{caratt-coniugati}
Two elements of $\mathcal{R}_\mathbb{K}$  are conjugated  over $\mathbb{K}$ if and only if they have the same $\mathbb{K}$-horizontal component and their $\mathbb{K}$-vertical components are conjugated over $\mathbb{K}$.
\end{prop}

\begin{proof}
Let $\lambda_1 =\alpha_1+ \beta_1, \lambda_2=\alpha_2+ \beta_2$ be elements of $\mathcal{R}_\mathbb{K}$ with their $\mathbb{K}$-decompositions.

If $\lambda_1, \lambda_2$ are conjugated  over $\mathbb{K}$ and $f(X) = \sum_{i=0}^d a_i X^i$ (with $a_d=1$) is their minimal polynomial, then $ \alpha_1 = H_{\mathbb{K}}(\lambda_1)= -\dfrac{a_{d-1}}{d} = H_{\mathbb{K}}(\lambda_2)= \alpha_2$. 

Denote $\alpha = \alpha_1= \alpha_2$, $\tilde{f}(X) = f(X+ \alpha)$ (which is a monic irreducible polynomial of degree $d$ in $\mathbb{K}[X]$) and $\{ \lambda_h = \alpha + \beta_h \ / \ h = 1, \cdots , d\}$ the conjugacy class of $\lambda_1$ and $\lambda_2$. 

Since $\tilde{f}(\beta_h) = f(\lambda_h) =0$ for every $h$, then $\tilde{f}(X)$ is the (monic) minimal polynomial of $\beta_1$ and $\beta_2$, which are therefore conjugated over $\mathbb{K}$.

Note that the coefficient of the term of degree $d-1$ of $\tilde{f}(X)$ is zero, because it is equal to $-\sum_{i=1}^d \beta_i = -\sum_{i=1}^d \lambda_i+ d\,\alpha = a_{d-1} +d\,\alpha = 0$.

For the converse, assume that $\beta_1, \beta_2$ are conjugated of degree $d$ over $\mathbb{K}$ and let $g(X)$ be their minimal polynomial. Then the polynomial $g(X - \alpha)$ is the minimal polynomial of $\alpha + \beta_1$ and of $\alpha + \beta_2$, which are therefore conjugated over $\mathbb{K}$.
\end{proof}

\begin{remdef}
We call the polynomial $\tilde{f}(X) = f(X - \dfrac{a_{d-1}}{d})$, introduced in \ref{caratt-coniugati},
\emph{the reduced form} of the monic irreducible polynomial 

$f(X) = \sum_{i=0}^{d-1} a_i X^i + X^d \in \mathbb{K}[X]$. 

This polynomial is monic, irreducible, of degree $d$ with coefficient of the term of degree $d-1$ equal to $0$, its roots are $\beta_1, \cdots , \beta_d$: the $\mathbb{K}$-vertical components of the roots of $f(X)$. 
Note that $\tilde{f}(X)$ is obtained by $f(X)$ by means of the translation associated to the $\mathbb{K}$-horizontal component of the roots of $f(X)$.

This transformation has a long and important role in history of theory of algebraic equations and it is the easiest one among the \emph{Tschirnhaus transformations} (see for instance \cite{Gar1927} and also \cite{Tignol2001} Ch.\,VI).
\end{remdef}

\begin{remdef}\label{involuzione} 
It is possible to define a natural involution on $\mathcal{R}_\mathbb{K}$ as follows. Every $\lambda \in \mathcal{R}_\mathbb{K}$ can be uniquely written as $\lambda = \alpha + \beta$ with $\alpha \in \mathbb{K}$ and  $\beta \in Ker(H_\mathbb{K})$. Since also $-\beta$ is in $Ker(H_\mathbb{K})$, we put $\overline{\lambda} = \alpha - \beta$, which is an element of $\mathcal{R}_\mathbb{K}$. 

Note that, if $char(\mathbb{K}) = 2$, then $\lambda = \overline{\lambda}$ for every $\lambda \in \mathcal{R}_\mathbb{K}$, otherwise $\lambda = \overline{\lambda}$ if and only if $\lambda \in \mathbb{K}$.

We call $\overline{\lambda}$ \emph{the} $\mathbb{K}$-\emph{involution of} $\lambda \in \mathcal{R}_\mathbb{K}$. 

It is easy to check that, when $\mathcal{R}_\mathbb{K} = \overline{\mathbb{K}}$ (e. g. in case of characteristic $0$), the $\mathbb{K}$-involution is an automorphism of the $\mathbb{K}$-vector space $\overline{\mathbb{K}}$, whose restriction to $\mathbb{K}$ is the identity.
\end{remdef}

\begin{prop}\label{car-non-2} 
Assume that the characteristic of $\mathbb{K}$ is not $2$.

\smallskip

a) If $\lambda =  \alpha + \beta \in \overline{\mathbb{K}}$ (with $\alpha \in \mathbb{K}$ and $\beta \in Ker(H_\mathbb{K})$) has degree $2$, then the unique conjugate of $\lambda$ is its $\mathbb{K}$-involution $\overline{\lambda} = \alpha - \beta$.

\smallskip

b) If $\theta_1, \cdots , \theta_l \in \overline{\mathbb{K}}$ are algebraic over ${\mathbb{K}}$ of degree at most $2$, then 

$\mathbb{K}(\theta_1, \cdots , \theta_l ) \subseteq \mathcal{R}_\mathbb{K}$ 
and \  $\overline{\theta_1 + \cdots + \theta_l} = \overline{\theta_1}  + \cdots + \overline{\theta_l}$.

\smallskip

c)  If $\lambda \in \overline{\mathbb{K}}$ has degree at most $2$, then $\overline{xy} = \overline{x} \, \overline{y}$ for every $x, y \in \mathbb{K}(\lambda)$.

\smallskip

d) If $\beta, \beta' \in Ker(H_\mathbb{K})$ have degree $2$ over $\mathbb{K}$, then $\beta \beta' \in \mathbb{K}$ if and only if they are linearly dependent over $\mathbb{K}$, otherwise $\beta \beta'$ has degree $2$ over $\mathbb{K}$ and it belongs to $Ker(H_\mathbb{K})$.

\smallskip

e) Let $\beta \in \overline{\mathbb{K}}$ has degree $2$ over $\mathbb{K}$, then $\beta \in Ker(H_\mathbb{K})$ if and only if $\beta \notin \mathbb{K}$ but $\beta^2 \in \mathbb{K}$.
\end{prop}

\begin{proof}
Let $\lambda = \alpha + \beta$ be as in (a) and let $f(X)$ be its monic minimal polynomial. Then the reduced form of $f(X)$ is  $\tilde{f}(X) = X^2 + N_{\mathbb{K}}(\beta)$. Since $\tilde{f}(X)$  has $\beta$ as root, also $- \beta$ is a root of $\tilde{f}(X)$ and so the roots of $\tilde{f}(X)$ are $\beta$ and $-\beta$ and $\lambda$ and $\overline{\lambda}$ are conjugated over $\mathbb{K}$ (remember \ref{caratt-coniugati}).

In order to prove (b), we have that the degree of $\mathbb{K}(\theta_1, \cdots , \theta_l)$  over  $\mathbb{K}$ is a power of $2$ and so for the degree of any its element too, therefore $\mathbb{K}(\theta_1, \cdots , \theta_l ) \subseteq \mathcal{R}_\mathbb{K}$.

We denote by $\theta_i = \alpha_i + \beta_i$ the $\mathbb{K}$-decomposition of $\theta_i$. 

By \ref{propr-D-K} (d) (with $\mathbb{M} =\mathbb{K}(\theta_1, \cdots , \theta_l ))$, we get: $\sum_i \beta_i \in  Ker(H_\mathbb{K})$. 

Hence  $\sum_i \alpha_i + \sum_i \beta_i$ is the the $\mathbb{K}$-decomposition of $\sum_i \theta_i$ and so:  

$\overline{\theta_1 + \cdots + \theta_l} =  \sum_i \alpha_i - \sum_i \beta_i =  \overline{\theta_1}  + \cdots + \overline{\theta_l}$.

Part (c) is trivial if the degree of $\lambda$ is $1$. If it has degree $2$, then $\lambda = \alpha + \beta$ with $\alpha \in \mathbb{K}$ and $\beta \in Ker(H_{\mathbb{K}})$, so $x = x_1 + x_2 \beta$ and $y = y_1 + y_2 \beta$ with $x_1, x_2, y_1, y_2 \in \mathbb{K}$. Hence we conclude by standard computations, because $\beta^2 \in \mathbb{K}$. 

If $\beta$ and $\beta'$ are as in (d), from (a), the conjugated of $\beta$ and $\beta'$ are respectively $-\beta$ and $-\beta'$. Hence $N_\mathbb{K}(\beta) = - \beta^2$ and $N_\mathbb{K}(\beta') = - {\beta'}^2$ are both in $\mathbb{K}$ and so $\beta \beta'$ is root of $X^2 -\beta^2 {\beta'}^2 \in \mathbb{K}[X]$.

The degree of $\beta \beta'$ is $1$ if and only if $\beta \beta' = t \in \mathbb{K}$, i. e. if and only if $\beta = \dfrac{t}{\beta'} = \dfrac{-t}{N_\mathbb{K}(\beta')} \beta'$ and so if and only if  $\beta, \beta'$ are linearly dependent over $\mathbb{K}$, because $\dfrac{-t}{N_\mathbb{K}(\beta')} \in \mathbb{K}$. 

Otherwise the degree of $\beta \beta'$ is $2$; then $X^2 -\beta^2 {\beta'}^2$ is its minimal polynomial and so $H_\mathbb{K}(\beta \beta') =0$. This concludes (d).

Finally, if $\beta$ has degree $2$ over $\mathbb{K}$, then $\beta \in \mathcal{R}_\mathbb{K}$ since $char(\mathbb{K}) \ne 2$ and so it belongs to $Ker(H_\mathbb{K})$ if and only if $\beta \notin \mathbb{K}$ and $\beta$ is root of a polynomial in $\mathbb{K}[X]$ of the form $X^2 + a_0$, i. e. if and only if $\beta \notin \mathbb{K}$ and $\beta^2 \in \mathbb{K}$.
\end{proof}

\section{Complete additive Jordan-Chevalley decomposition.}

\begin{notas}
From now on, as in \cite{DoPe2017}, $M$ is a fixed matrix in $M_n(\mathbb{K})$ with minimal polynomial $m(X) =m_1(X)^{\mu_1} \cdots m_r(X)^{\mu_r}$
where $\mu_1, \cdots , \mu_r > 0$ and $m_1(X), \cdots , m_r(X)$ are mutually distinct irreducible polynomials in $\mathbb{K}[X]$ of degrees $d_1, \cdots , d_r$ respectively, $\mathbb{F}$  is the splitting field of $m(X)$ over $\mathbb{K}$ and $\mathbb{K}^\dag$ is the fixed subfield of $\mathbb{F}$ with respect to the group  $Aut(\mathbb{F}/\mathbb{K})$ of $\mathbb{K}$-automorphisms of $\mathbb{F}$.
\end{notas}

\begin{defi}\label{def-matr-CR}
We say that the matrix $M \in M_n(\mathbb{K})$ is $\mathbb{K}$-\emph{regular}, if all its eigenvalues are in $\mathcal{R}_\mathbb{K}$. 

This is equivalent to say that no irreducible component over $\mathbb{K}$ of its minimal polynomial has degree multiple of the characteristic of $\mathbb{K}$ 
\end{defi}

\begin{rem}\label{estems-galois}
If $M$ is $\mathbb{K}$-regular, then every irreducible component $m_i(X)$ of $m(X)$ (as above) has $d_i$ mutually distinct roots $\lambda_{i 1}, \cdots , \lambda_{i d_i}$, which are conjugated over $\mathbb{K}$. 

Hence the extension $\mathbb{F} / \mathbb{K}$ is a Galois extension, i. e. $\mathbb{K}^\dag =  \mathbb{K}$.

Indeed each irreducible component of $m(X)$ is  separable and $\mathbb{F}$ is also normal, because it is a splitting field (see for instance  \cite{Lang2002} Ch. VI, Thm. 1.2).
\end{rem}

\begin{recalls}\label{richiamiDoPe2017}
Every matrix $M \in M_n(\mathbb{K})$ has a unique decomposition, called \emph{additive Jordan-Chevalley decomposition}, $M= S(M) + N(M)$, with $S(M)$ \emph{semisimple} (i.\,e. $S(M)$ is similar over $\overline{\mathbb{K}}$ to a diagonal matrix),
$N(M)$ nilpotent and $S(M)N(M) = N(M)S(M)$ where both $S(M)$ and $N(M)$ have coefficients in $\mathbb{K}^\dag$ and are polynomial expressions of $M$.

Moreover the matrix $S(M)$ has a unique \emph{Frobenius decomposition}. If $M$ is $\mathbb{K}$-regular, this decomposition is 
\begin{center}
$S(M) =\sum_{i=1}^r \sum_{j=1}^{d_i} \lambda_{ij} C_{ij}(M)$,
\end{center}
 where $\cup_{i=1}^r \{\lambda_{i1}, \cdots , \lambda_{i d_i} \}$ is the set of all distinct eigenvalues of $M$ (and of $S(M)$), arranged in conjugacy classes, and 
 
 $\{C_{ij}(M)\}_{i= 1, \cdots , r}^{j= 1 , \cdots , d_i}$ is a \emph{Frobenius system}, 

i.e. $C_{ij}(M) C_{i'j'}(M) = C_{ij}(M)$ if $(i,j)=(i',j')$ and $C_{ij}(M) C_{i'j'}(M) = 0$ 
otherwise, 

with the further condition $\sum_{i=1}^r \sum_{j=1}^{d_i} C_{ij}(M) = I_n$ (the identity matrix of order $n$). 

The matrices $C_{ij}(M)$'s (uniquely determined by $M$) have coefficients in $\mathbb{F}$ and are polynomial expressions of $M$ of degree strictly less than  $deg (m(X))$.

For all previous facts we refer to \cite{DoPe2017} \S 1.
\end{recalls}

\begin{defi}\label{Def-CAJC-Dec}
We say that a $\mathbb{K}$-regular matrix is $\mathbb{K}$-\emph{vertical}, if all its eigenvalues are in $Ker(H_\mathbb{K})$.

A \emph{complete additive Jordan-Chevalley decomposition over} $\mathbb{K}$ \emph{of} $M \in M_n(\mathbb{K})$ is any additive decomposition 
\begin{center}
$M = H + V + N$
\end{center}
where $H, V, N \in M_n(\mathbb{K})$ are mutually commuting, $H$ is diagonalizable (over $\mathbb{K}$), $V$ is $\mathbb{K}$-vertical and semisimple and $N$ is nilpotent.
\end{defi}

\begin{thm}\label{Delta-Sigma-N}
Assume that $M \in M_n(\mathbb{K})$ is $\mathbb{K}$-regular.

Then there exists a unique complete additive Jordan-Chevalley decomposition over  $\mathbb{K}$ of $M$
\begin{center}
$M = H(M) + V(M) + N(M).$
\end{center}
Moreover $H(M),  V(M), N(M)$ are polynomial expressions of $M$ with coefficients in $\mathbb{K}$ and $H(M) +  V(M)$ and $N(M)$ coincide respectively with the semisimple and the nilpotent components of the additive Jordan-Chevalley decomposition of $M$.

We call $H(M)$ and $V(M)$ respectively $\mathbb{K}$-\emph{horizontal} and $\mathbb{K}$-\emph{vertical components} \emph{of} $M$. 
\end{thm}

\begin{proof}
Remember the notations introduced in \ref{richiamiDoPe2017}. From \ref{estems-galois} the semisimple and nilpotent components of $M$ have both coefficients in $\mathbb{K}$.

We denote: $\alpha_i= \dfrac{\lambda_{i1} + \cdots + \lambda_{id_i}}{d_i} = H_\mathbb{K}(\lambda_{ij}) \in \mathbb{K}$ for every $i=1, \cdots , r$ and every $j=1, \cdots , d_i$. 

Hence, by Proposition \ref{decomp-traccia}, we can write, in a unique way, $\lambda_{ij} = \alpha_i + \beta_{ij}$ with $\beta_{ij} \in Ker(H_\mathbb{K})$ for every $i,j$.

Now let $S(M) = \sum_{i=1}^r  \sum_{j=1}^{d_i} \lambda_{ij} C_{ij}(M)$ be the semisimple component of the Jordan-Chevalley decomposition of $M$ with its Frobenius decomposition.

By decomposing each $\lambda_{ij}$ as $\alpha_i + \beta_{ij}$ we get
\begin{center}
$S(M)= H(M) + V(M)$
\end{center}
where 
$H(M) = \sum_{i=1}^r \alpha_i \sum_{j= 1}^{d_i}  C_{ij}(M)$
and
$V(M)= \sum_{i=1}^r
 \sum_{j=1}^{d_i} \beta_{ij} C_{ij}(M)$;
 
 both matrices are polynomial expressions of $M$ and, arguing as in proof of 
 
 \cite{DoPe2017} Thm. 1.6, they have coefficients in the fixed field $\mathbb{K}^\dag = \mathbb{K}$ (remember \ref{estems-galois}).
 
$H(M)$ and $V(M)$ are respectively diagonalizable over $\mathbb{K}$ and over $\mathbb{F}$ with eigenvalues $\alpha_i$'s and $\beta_{ij}$'s by \ref{richiamiDoPe2017} and by \cite{DoPe2017} Prop. 1.9.
 
This allows to conclude about the existence of a complete additive Jordan-Chevalley decomposition over $\mathbb{K}$ of $M$  in terms of polynomial expressions and about its relations with the Jordan-Chevalley decomposition.

Now let $M = H' + V' + N'$ be any other complete additive Jordan-Chevalley decomposition over  $\mathbb{K}$ of $M$.

$H' + V'$ is semisimple, because it is sum of commuting semisimple matrices, therefore, by uniqueness of the Jordan-Chevalley decomposition, $N'= N(M)$ and $H' + V' = H(M) + V(M)$. This implies $H(M) - H' + V(M) =  V'$. Since $H',  V'$ commute with $M$ and so with $H(M), V(M)$ and since the four matrices are semisimple, every  eigenvalue $\sigma'$ of $V'$ can be written as $\sigma' = \delta - \delta' + \sigma$ with $\delta, \delta', \sigma$ eigenvalues of $H(M), H', V(M)$ respectively. But $\sigma', \delta - \delta' + \sigma$ are in $\mathcal{R}_\mathbb{K}$ by \ref{prime-propr-C_K} (c) and we can conclude $\delta - \delta' =0$ and $\sigma = \sigma'$ by \ref{decomp-traccia} (because $\delta - \delta' \in \mathbb{K}$ and $\sigma, \sigma' \in Ker(H_\mathbb{K})$). Therefore $H' = H(M)$ and $V' = V(M)$.
\end{proof}

\begin{rems} 
a) From the proof of the previous Theorem we get that

 $H(M)= \sum_{i=1}^r \alpha_i \sum_{j= 1}^{d_i}  C_{ij}(M)$ with (possibly repeated) eigenvalues $\alpha_1, \cdots , \alpha_r$
and

$V(M) = \sum_{i=1}^r
 \sum_{j=1}^{d_i} \beta_{ij} C_{ij}(M)$ with (possibly repeated) eigenvalues $\beta_{ij}$, $1 \le i \le r$, $1 \le j \le d_i$

\smallskip

b) Up to summing the matrices $C_{ij}(M)$'s having possibly the same $\alpha_i$'s or the same $ \beta_{ij}$'s, the previous give the Frobenius decompositions of $H(M)$ and of $V(M)$.
\end{rems}

\section{Fine Frobenius decomposition of a matrix.}

\begin{defis} \label{splitting-bound}
Let $\mathbb{L}/\mathbb{K}$ be an extension of fields.

a) The \emph{splitting bound over} $\mathbb{L}$ of a polynomial $f(X) \in \mathbb{K}[X]$ is the maximum degree among its irreducible factors in $\mathbb{L}[X]$.

b) The \emph{splitting bound over} $\mathbb{L}$ of the matrix $M \in M_n(\mathbb{K})$ is the splitting bound over $\mathbb{L}$ of its minimal polynomial.
\end{defis}  

\begin{rems}\label{oss-dopo-splitting-bound}
a) A polynomial has splitting bound $1$ over a field $\mathbb{L}$ if and only if $\mathbb{L}$ contains its splitting field. Hence the matrix $M \in M_n(\mathbb{K})$ has splitting bound $1$ over $\mathbb{L}$ if and only if its eigenvalues are all in $\mathbb{L}$.

\smallskip

b) If $\mathbb{K}$ has characteristic different from $2$ and the matrix $M \in M_n(\mathbb{K})$ has splitting bound at most $2$ over $\mathbb{K}$, then it is easy to get that $M$ is $\mathbb{K}$-regular.
\end{rems}

\begin{defi}\label{def-fine-Frob-Dec}
We call \emph{fine Frobenius decomposition} of the matrix $M \in M_n(\mathbb{K})$ any decomposition
\begin{center}
$M = \sum_{i=1}^s \gamma_i A_i - \sum_{j=1}^t \dfrac{\alpha_j}{n_j} B_j^2 + \sum_{j=1}^t B_j$\end{center}

such that 
$A_1, \cdots A_s, B_1, \cdots , B_t$ is a non-empty family of matrices in $M_n(\mathbb{K}) \setminus \{0\}$

and, for all possible indices, all the following conditions hold:

$\gamma_i, \alpha_j, n_j \in  \mathbb{K}$, but $\sqrt{-n_j} \notin \mathbb{K}$ (where $\sqrt{-n_j}$ denotes an arbitrary root in $\overline{\mathbb{K}}$ of $X^2 +n_j$);

$\gamma_i \ne 0$ and $\gamma_i \ne \gamma_h$ as soon as $i \ne h$;

$\alpha_j+\sqrt{-n_j} \ne \alpha_l \pm \sqrt{-n_l}$ as soon as $j \ne l$;

$A_i A_h = \delta_{ih}A_i$;

$A_i B_j = B_jA_i = 0$;

$B_j B_l =0$ as soon as $j \ne l$;

$B_j^3 = - n_j B_j$.

We say that the matrices $A_1, \cdots , A_s, B_1, \cdots , B_t$ are \emph{fine Frobenius covariants of} $M$; more precisely $A_1, \cdots , A_s$ are \emph{horizontal covariants of} $M$ and $B_1, \cdots , B_t$ are \emph{vertical covariants} of $M$.
\end{defi}

\begin{lemma}
Assume that $\mathbb{K}$ has characteristic different from $2$ and that $M \in M_n(\mathbb{K})$ has a fine Frobenius decomposition:
\begin{center}
$
M = \sum_{i=1}^s \gamma_i A_i - \sum_{j=1}^t \dfrac{\alpha_j}{n_j} B_j^2 + \sum_{j=1}^t B_j.
$
\end{center}
Then

a) $M$ is semisimple, nonzero and its mutually distinct nonzero eigenvalues are $\gamma_1, \cdots , \gamma_s$, \, $\alpha_1 + \sqrt{-n_1}, \, \alpha_1 - \sqrt{-n_1}, \cdots , \alpha_t + \sqrt{-n_t}, \, \alpha_t - \sqrt{-n_t}$.

b) $\overline{\mathbb{K}}^n = [\oplus_{i=1}^s Im (A_i)] \oplus [\oplus_{j=1}^t Im(B_j)] \oplus Ker(M)$,

where for all possible indices $i,j$ we have 

$Im(A_i) = Ker(M- \gamma_i I_n)$; 

$Im(B_j) = Ker(M-(\alpha_j+ \sqrt{-n_j})I_n) \oplus Ker(M-(\alpha_j - \sqrt{-n_j})I_n)$; 

$Ker(M) = [\cap_{i=1}^s Im(A_i)] \cap [\cap_{j=1}^t Ker(B_j)]$; 

$Ker(M-(\alpha_j+ \sqrt{-n_j})I_n)= Ker(B_j -\sqrt{-n_j} I_n)$; 

$Ker(M-(\alpha_j - \sqrt{-n_j})I_n)= Ker(B_j +\sqrt{-n_j} I_n)$.

c) For every $i= 1, \cdots , s$, \, $A_i$ acts as the identity map on 
$Ker(M-\gamma_i I_n)$ and as the null map on any other eigenspace of $M$. Hence $A_i$ is uniquely determined by $M$.

d)   For every $j=1, \cdots , t$, \,  $B_j$ acts as the multiplication by $-\sqrt{-n_j}$ 

on $Ker(M-(\alpha_j-\sqrt{-n_j})I_n)$ and as the null map on any other eigenspace of $M$. Hence $B_j$ is uniquely determined by $M$.

e) For every $j=1, \cdots , t$, $\lambda_j = \alpha_j + \sqrt{-n_j}$ has degree $2$ over $\mathbb{K}$ and its conjugated is  $\overline{\lambda}_j = \alpha_j - \sqrt{-n_j}$. Hence  $\alpha_j = H_\mathbb{K}(\lambda_j)$, $n_j = -V_\mathbb{K}(\lambda_j)^2$ and $M$ has splitting bound at most $2$ over $\mathbb{K}$.
\end{lemma}

\begin{proof} 
These facts follow from standard arguments of linear algebra; the proofs of similar statements can be found for instance in \cite{YTT2011} Ch.2 (see also \cite{DoPe2017} Lemma 1.2).
\end{proof}

\begin{prop}\label{dec-fine-unica}
Assume that $\mathbb{K}$ has characteristic different from $2$. 

Then $M \in M_n(\mathbb{K})$ has a fine Frobenius decomposition
if and only if it is semisimple, nonzero and its splitting bound over $\mathbb{K}$ is at most $2$. 

If this is the case, the decomposition is unique up to to the order of the addends and it is given by
\begin{center}
$
M = \sum_{i=1}^s \gamma_i \mathcal{A}_i(M) + \sum_{j=1}^t \dfrac{H_\mathbb{K}(\lambda_j)}{V_\mathbb{K}(\lambda_j)^2} \mathcal{B}_{j}(M)^2 + \sum_{j=1}^t \mathcal{B}_{j}(M)
$
\end{center}
where $\gamma_1 , \cdots , \gamma_s$ are the distinct nonzero eigenvalues of $M$ belonging to $\mathbb{K}$;  

$\lambda_1, \  \overline{\lambda_1 }, \cdots , \lambda_t, \ \overline{\lambda_t} $ are the distinct eigenvalues of $M$ not in $\mathbb{K}$;

$\mathcal{A}_i(M) = C_{i1}(M)$ ($i = 1, \cdots , s$) matrices in $M_n(\mathbb{K})$;

$\mathcal{B}_{j}(M)=   V_\mathbb{K}(\lambda_j)\, [C_{s+j, 1}(M) -   \overline{C}_{s+j, 1}(M)]$ 
($ j = 1, 	\cdots, t$) matrices in $M_n(\mathbb{K})$. 
\end{prop}

\begin{proof}
The previous lemma gives easily one implication and the uniqueness.

For the converse, we consider all irreducible components over $\mathbb{K}$, $m_i(X) \ne X$, of the minimal polynomial of $M$.

Up to reorder them, we can assume that the first $s$ polynomials have degree $1$ and that the last $t$ polynomials have degree $2$. We denote by $\gamma_1 , \cdots , \gamma_s$ the distinct nonzero eigenvalues belonging to $\mathbb{K}$ and, remembering \ref{car-non-2} (a), by $\lambda_j = H_\mathbb{K}(\lambda_j)  + V_\mathbb{K}(\lambda_j)$, $\overline{\lambda}_j = H_\mathbb{K}(\lambda_j)  - V_\mathbb{K}(\lambda_j)$ with $H_\mathbb{K}(\lambda_j) \in \mathbb{K}$ and $V_\mathbb{K}(\lambda_j) \in Ker(H_\mathbb{K}) \setminus \{0\}$, $1 \le j \le t$, the distinct eigenvalues not in $\mathbb{K}$.

We can write the semisimple component of $M$ as 
\begin{center}
$S(M) = \sum_{i=1}^s \gamma_i C_{i1}(M) + \sum_{i=1}^t [\lambda_i C_{s+i, 1}(M) + \overline{\lambda}_i C_{s+i, 2}(M)]$.
\end{center}

As shown in \cite{DoPe2017} Prop.\,1.5, $C_{i1}(M)$ has coefficients in $\mathbb{K}$ and $C_{s+i,1}(M), C_{s+i,2}(M)$ have coefficients in $\mathbb{K}(\lambda_i)$. 

As checked in the proof of \cite{DoPe2017} Thm.\,1.6, 

$\sigma(C_{s+i,1}(M))= C_{s+i,2}(M)$, where $\sigma \in Aut(\mathbb{F} / \mathbb{K})$ maps $\lambda_i$ into $\overline{\lambda}_i$. 

Now $[\mathbb{K}(\lambda_i) : \mathbb{K}] = 2$, hence the restriction of such a $\sigma$ to $\mathbb{K}(\lambda_i)$ is the $\mathbb{K}$-involution and so: $C_{s+i,2}(M) = \overline{C}_{s+i,1}(M)$.

Therefore we get:
\begin{center}
$S(M) = \sum_{i=1}^s \gamma_i C_{i1}(M) + \sum_{i=1}^t [\lambda_i C_{s+i, 1}(M) + \overline{\lambda}_i \, \overline{C}_{s+i, 1}(M)].$
\end{center}
We conclude by setting $\mathcal{A}_i(M) = C_{i1}(M)$ ($1 \le i \le s$) and 

$\mathcal{B}_{j}(M) = V_\mathbb{K}(\lambda_j) [C_{s+j, 1}(M) -  \overline{C}_{s+j, 1}(M)]$ ($ 1 \le j \le t$).
\end{proof}

\begin{rems}\label{oss-dopo-dec-fine-unica}
a) Note that, by the previous Proposition, the fine Frobenius covariants of $M$ are polynomial functions of $M$, expressed by means of its so-called \emph{Frobenius covariants} and are uniquely determined by $M$;
the fine Frobenius decomposition is the so-called \emph{Frobenius decomposition} if and only if $M$ is diagonalizable over $\mathbb{K}$ (see \cite{DoPe2017}).

Note also that $\dfrac{\mathcal{B}_j(M)^2}{V_\mathbb{K}(\lambda_j)^2}$ is idempotent for every $j$.

\smallskip

b) In the expression of $M$ in the previous Proposition, the possible eigenvalue $0$ does not appear explicitly and we can retrieve it as follows.
We set:
\begin{center}
$\gamma_0 = 0 \mbox{ and } \mathcal{A}_0(M) = I_n - \sum_{i=1}^s \mathcal{A}_i(M) - \sum_{j=1}^t \dfrac{\mathcal{B}_j(M)^2}{V_\mathbb{K}(\lambda_j)^2}.$
\end{center} 
Note that $\mathcal{A}_0(M) \ne 0$ if and only if $\gamma_0=0$ is an eigenvalue of $M$.

Therefore from now on, we can write the fine Frobenius decomposition of $M$ as:
\begin{center}
$
M = \sum_{i=0}^s \gamma_i \mathcal{A}_i(M) +  \sum_{j=1}^t \dfrac{H_\mathbb{K}(\lambda_j)}{V_\mathbb{K}(\lambda_j)^2} \mathcal{B}_{j}(M)^2 + \sum_{j=1}^t \mathcal{B}_{j}(M).
$
\end{center}
\end{rems}

\section{Some formulas for series of powers of matrices on a valued field.} 

\begin{rem} In this section we want to assume that $\mathbb{K}$ is endowed with a  \emph{absolute value} $|\cdot|$. We call such a pair $(\mathbb{K}, |\cdot|)$ a \emph{valued field}. We refer for instance to \cite{Jac1989} Ch.\,9, to \cite{War1989} Ch.\,III, Ch.\,IV, to 
\cite{Lang2002} Ch.\,XII and to \cite{Lor2008} Ch.\,23 for more information. 

Let $(\mathbb{K}, |\cdot|)$ be a valued field. The absolute value over $\mathbb{K}$ extends in a unique way to its \emph{completion} $\mathbb{K}^c$; this one extends in a unique way to an absolute value over a fixed algebraic closure $\overline{\mathbb{K}^c}$ of $\mathbb{K}^c$ and finally the last extends in a unique way to the completion $(\overline{\mathbb{K}^c})^c$  (see for instance to \cite{Lor2008} Thm.\,2 p.\,48, Ostrowski's Thm. p.\,55 and Thm.\,4' p.\,60). We denote all extensions always by the same symbol $|\cdot|$.

Note that the field $\{ \alpha \in \overline{\mathbb{K}^c} \ / \ \alpha \mbox{ is algebraic over } \mathbb{K} \}$ is the unique algebraic closure of $\mathbb{K}$ contained in $\overline{\mathbb{K}^c}$ and therefore it can be identified with $\overline{\mathbb{K}}$. 

By restriction, we get  an absolute value over $\overline{\mathbb{K}}$, extending the absolute value of $\mathbb{K}$.
\end{rem}

\begin{lemma}\label{completamento}
Let $(\mathbb{K}, |\cdot|)$ be a valued field.

a) If $\mathbb{K}$ is algebraically closed, then its completion $\mathbb{K}^c$ (with respect to $|\cdot|$) is algebraically closed too.

b) $(\overline{\mathbb{K}^c})^c$ is algebraically closed and complete with respect to the unique extension to it of $|\cdot|$.

\end{lemma}
\begin{proof}
The proof of (a) follows easily from Ostrowski's Theorem in archimedean case (see \cite{Lor2008} p.\,55), while for the non-archimedean case we refer to \cite{Lor2008} 24.15 p.\,316.

Finally (b) follows directly from (a).
\end{proof}

\begin{remsdefs}\label{recall-valued-fields}
a) Let $f(X) = \sum_{m=0}^{+ \infty} a_m X^m$ be a series with coefficients in the valued field $(\mathbb{K}, |\cdot|)$ and $R_f \in \mathbb{R} \cup \{+ \infty\}$ be the \emph{radius of convergence} of the associated real series $\sum_{m=0}^{+ \infty} |a_m| X^m$.

We denote by $\Omega_{f,\mathbb{K}}$ the subset of matrices  $A \in M_n(\mathbb{K})$,  whose \emph{the spectral radius} is strictly less than $R_f$. 

Note that if $A \in \Omega_{f,\mathbb{K}}$, then $f(A) \in M_n(\mathbb{K}^c)$ (see for instance \cite{DoPe2017} 3.1(c)).

Let $\widehat{\Omega}_{f,\mathbb{K}}$ be the subset of $\Omega_{f,\mathbb{K}}$ of matrices $A$ whose eigenvalues $\lambda = \alpha + \beta$ (with their $\mathbb{K}$-decompositions) satisfy:

i) $|\alpha| + |\beta| < R_f$, if the absolute value is archimedean

ii) $max(|\alpha|, |\beta|) < R_f$,  if the absolute value is non-archimedean.

\smallskip

b) Let $(\mathbb{K}, |\cdot|)$ be a valued field. In $\overline{\mathbb{K}}$ we choose a square root of $-1$, denoted by $\sqrt{-1}$ and let $f(X) = \sum_{m=0}^{+ \infty} a_m X^m$ be a series as above.

Let $\lambda =\alpha + \beta$ be in $\mathcal{R}_\mathbb{K}$ together with its $\mathbb{K}$-decomposition.

We consider the formal series 
$$\mathcal{R} f(\lambda) = \sum_{m=0}^{+\infty} a_m \sum_{h=0}^{\lfloor m/2\rfloor} {m \choose 2h} \alpha^{m-2h} \beta^{2h}$$
$$\mathcal{I} f(\lambda) = -\sqrt{-1}\, [\sum_{m=1}^{+\infty} a_m \sum_{h=1}^{\lfloor (m+1)/2\rfloor} {m \choose 2h-1} \alpha^{m-2h+1} \beta^{2h-1}].$$
It is easy to check that, if $|\lambda|, |\alpha|, |\beta| < R_f$, then 

$f(\lambda)$, $\mathcal{R} f(\lambda)$, $\mathcal{I} f(\lambda)$ converge in $(\overline{\mathbb{K}})^c$ and
\begin{center}
$
f(\lambda) = \mathcal{R} f(\lambda) + \sqrt{-1} \, \mathcal{I} f(\lambda).
$
\end{center}
\end{remsdefs}

\begin{prop}\label{decf(M)}
Let $f(X) = \sum_{m=0}^{+ \infty} a_m X^m$ be a series with coefficients in the valued field $(\mathbb{K}, |\cdot|)$ of characteristic different from $2$ and $M \in \widehat{\Omega}_{f, \mathbb{K}} \setminus \{0\}$ be a semisimple matrix with splitting bound  at most $2$ over $\mathbb{K}$ and with fine Frobenius decomposition:

$
M = \sum_{i=0}^s \gamma_i \mathcal{A}_i(M) + \sum_{j=1}^t \dfrac{H_\mathbb{K}(\lambda_j)}{V_\mathbb{K}(\lambda_j)^2} \mathcal{B}_{j}(M)^2 + \sum_{j=1}^t \mathcal{B}_{j}(M)
$
as in \ref{dec-fine-unica} and in \ref{oss-dopo-dec-fine-unica} (b).

Then
\begin{center}
$
f(M) = \sum_{i=0}^s f(\gamma_i) \mathcal{A}_i(M) + \sum_{j=1}^t [ \dfrac{\mathcal{R}f(\lambda_j)}{V_\mathbb{K}(\lambda_j)^2} \, \mathcal{B}_j(M)^2  + \sqrt{-1} \, \dfrac{\mathcal{I}f(\lambda_j)}{V_\mathbb{K}(\lambda_j)} \mathcal{B}_j(M)],
$
\end{center}

with $f(\gamma_i), \ \dfrac{\mathcal{R}f(\lambda_j)}{V_\mathbb{K}(\lambda_j)^2}, \ \sqrt{-1} \, \dfrac{\mathcal{I}f(\lambda_j)}{V_\mathbb{K}(\lambda_j)} \ \in \mathbb{K}^c$ for every $i,j$.
\end{prop}

\begin{proof}
We pose $\alpha_j = H_\mathbb{K}(\lambda_j)$, $n_j = - V_\mathbb{K}(\lambda_j)^2$ , so that we can write: 

$\lambda_j = \alpha_j + \sqrt{- n_j}$.

From $\mathcal{B}_j(M)^3 = - n_j \mathcal{B}_j(M)$ we get:

$\mathcal{B}_j(M)^{2k} = -(n_j)^{k-1} \mathcal{B}_j(M)^2$ and
$\mathcal{B}_j(M)^{2k-1} = -(n_j)^{k-1} \mathcal{B}_j(M)$ for every $k \ge 1$.

Therefore standard computations allow to get, for every $m \ge 1$:
\begin{multline*}
[ \dfrac{H_\mathbb{K}(\lambda_j)}{V_\mathbb{K}(\lambda_j)^2} \mathcal{B}_j(M)^2 + \mathcal{B}_j(M)]^m =\\
[\sum_{h=0}^{\lfloor m/2 \rfloor} \dfrac{{m \choose 2h} \alpha_j^{m-2h} (\sqrt{-n_j})^{2h}}{-n_j} ] \mathcal{B}_j(M)^2+\\
[\sum_{h=1}^{\lfloor (m+1)/2 \rfloor} \dfrac{{m \choose 2h-1} \alpha_j^{m-2h+1} (\sqrt{-n_j})^{2h-1}}{\sqrt{-n_j}}]\mathcal{B}_j(M).
\end{multline*}

Hence:
$$
f(M) = a_o I_n + \sum_{m=1}^{+ \infty} a_m [\sum_{i=0}^s \gamma_i \mathcal{A}_i(M) + \sum_{j=1}^t (\dfrac{\alpha_j}{-n_j} \mathcal{B}_j(M)^2 + \mathcal{B}_j(M))]^m.
$$
Remembering the properties of the various matrices we get that the last is equal to
\begin{multline*}
a_o I_n + \sum_{m=1}^{+ \infty} a_m [\sum_{i=0}^s \gamma_i^m \mathcal{A}_i(M)] + \sum_{m=1}^{+ \infty} a_m  [\dfrac{\alpha_j}{-n_j} \mathcal{B}_j(M)^2 + \mathcal{B}_j(M)]^m =\\
a_0 [I_n - \sum_{i=0}^s \mathcal{A}_i(M) + \sum_{j=1}^t \dfrac{\mathcal{B}_i(M)^2}{n_j}] + \sum_{i=0}^s f(\gamma_i) \mathcal{A}_i(M) + \\
\sum_{j=1}^t[\dfrac{\sum_{m=0}^{+\infty} a_m \sum_{h=0}^{\lfloor m/2\rfloor} {m \choose 2h} \alpha_j^{m-2h} \sqrt{-n_j}^{2h}}{-n_j}] \mathcal{B}_j(M)^2 + \\
\sum_{j=1}^t[\dfrac{\sum_{m=1}^{+\infty} a_m \sum_{h=1}^{\lfloor (m+1)/2\rfloor} {m \choose 2h-1} \alpha_j^{m-2h+1} \sqrt{-n_j}^{2h-1}}{\sqrt{-n_j}}] \mathcal{B}_j(M).
\end{multline*}
Now, remembering \ref{oss-dopo-dec-fine-unica} (b) and the definitions of the various matrices, we get the expression of $f(M)$ in the statement.

Note that the expressions of $\mathcal{R}f(\lambda_j)$ and of $\dfrac{\mathcal{I}f(\lambda_j)}{V_\mathbb{K}(\lambda_j)}$ are invariant up to change $\lambda_j$ with $\overline{\lambda}_j$.
\end{proof}

\begin{examples}\label{exp-cos'}
Assume that $(\mathbb{K}, |\cdot |)$ is a valued field with $char(\mathbb{K}) =0$. Then the restriction to fundamental field $\mathbb{Q}$ of  $|\cdot |$ is equivalent either to the usual euclidean absolute value, or to the trivial absolute value, or to a p-adic absolute value for some prime number $p$ (see for instance \cite{Lor2008} Ch.\,23, Thm.\,1). 

In correspondence to this three cases, we call $(\mathbb{K}, |\cdot |)$ a valued field of \emph{archimedean type}, or of \emph{trivial type}, or of \emph{p-adic type}, respectively.

In all cases we can define as power series, as in ordinary real case, the exponential function, the sinus, the cosinus,  the hyperbolic sinus and the hyperbolic cosinus. These series have the same radius of convergence: $R=1$ if the absolute value is of trivial type, $R= +\infty$ if the absolute value is of archimedean type and  $R= (\dfrac{1}{p})^{^{\frac{1}{p-1}}}$ if the absolute value is of p-adic type (see for instance \cite{Schik1984} pp.70--72).

We set $\widehat{\Omega}_{\mathbb{K}} = \widehat{\Omega}_{exp, \mathbb{K}}$.

Let $\lambda = H_\mathbb{K}(\lambda) + V_\mathbb{K}(\lambda) \in \overline{\mathbb{K}}$ with its $\mathbb{K}$-decomposition. Then for $f(\lambda) = \exp(\lambda)$ we have

$
\mathcal{R} f(\lambda) = \exp(H_\mathbb{K}(\lambda) ) \cosh(V_\mathbb{K}(\lambda) )
$ and 
$
\mathcal{I} f(\lambda) = \dfrac{\exp(H_\mathbb{K}(\lambda) ) \sinh(V_\mathbb{K}(\lambda) )} {\sqrt{-1}}
$.

So: $\exp(\lambda) = \exp(H_\mathbb{K}(\lambda) ) [\cosh(V_\mathbb{K}(\lambda) ) + \sinh(V_\mathbb{K}(\lambda) )]$.

Now if $M \in \widehat{\Omega}_{\mathbb{K}} \setminus \{0\}$ is semisimple and has splitting bound at most $2$ over $\mathbb{K}$, then from the previous Proposition we get:
\begin{multline*}
\exp(M)= \sum_{i=0}^s \exp(\gamma_i) \mathcal{A}_i(M) + \sum_{j=1}^t \dfrac{\exp(H_\mathbb{K}(\lambda_j)) \cosh(V_\mathbb{K}(\lambda_j) )}{V_\mathbb{K}(\lambda_j)^2} \mathcal{B}_j(M)^2 + \\
\sum_{j=1}^t \dfrac{\exp(H_\mathbb{K}(\lambda_j)) \sinh(V_\mathbb{K}(\lambda_j))}{V_\mathbb{K}(\lambda_j)} \mathcal{B}_j(M).
\end{multline*}

Analogously we can get the formulas for $\cos(M)$ and $\sin(M)$; for instance 

if $M \in \widehat{\Omega}_{\mathbb{K}} \setminus \{0\}$ is semisimple and has splitting bound at most $2$ over $\mathbb{K}$, then 
\begin{multline*}
\cos(M)= \sum_{i=0}^s \cos(\gamma_i) \mathcal{A}_i(M) + \sum_{j=1}^t \dfrac{\cos(H_\mathbb{K}(\lambda_j)) \cos(V_\mathbb{K}(\lambda_j))}{V_\mathbb{K}(\lambda_j)^2} \mathcal{B}_j(M)^2 + \\
\sum_{j=1}^t \dfrac{\sin(H_\mathbb{K}(\lambda_j)) \sin(H_\mathbb{K}(\lambda_j))}{V_\mathbb{K}(\lambda_j)} \mathcal{B}_j(M).
\end{multline*}
\end{examples}

\begin{cor}\label{CAJC-dec-f-M}
Let $f(X) = \sum_{m=0}^{+ \infty} a_m X^m$ be a series with coefficients in the valued field $(\mathbb{K}, |\cdot|)$ of characteristic different from $2$ and $M \in \widehat{\Omega}_{f, \mathbb{K}} \setminus \{0\}$ be a semisimple matrix with splitting bound at most $2$ over $\mathbb{K}$ and with fine Frobenius decomposition

$
M = \sum_{i=0}^s \gamma_i \mathcal{A}_i(M) + \sum_{j=1}^t \dfrac{H_\mathbb{K}(\lambda_j)}{V_\mathbb{K}(\lambda_j)^2} \mathcal{B}_{j}(M)^2 + \sum_{j=1}^t \mathcal{B}_{j}(M)
$. 

Then
$f(M)$ is semisimple with spitting bound at most $2$ over $\mathbb{K}^c$
and its complete additive Jordan-Chevalley decomposition over $\mathbb{K}^c$ is 
\begin{center}
$f(M) = H(f(M)) + V(f(M))$
\end{center}

where

$H(f(M))= \sum_{i=0}^s f(\gamma_i) \mathcal{A}_i(M)  
+ \sum_{j=1}^t \mathcal{R}f (\lambda_j) \dfrac{\mathcal{B}_j(M)^2}{V_\mathbb{K}(\lambda_j)^2} 
$ 

whose (possibly repeated) eigenvalues are $f(\gamma_i)$ and $\mathcal{R}f (\lambda_j)$ for every admissible $i,j$

$V(f(M))= \sqrt{-1}\,\sum_{j=1}^t \dfrac{\mathcal{I}f(\lambda_j)}{V_\mathbb{K}(\lambda_j)} \mathcal{B}_j(M)$ 

whose (possibly repeated) eigenvalues are $\pm \sqrt{-1}\,\mathcal{I}f(\lambda_j)$ and possibly $0$.
\end{cor}

\begin{proof}
It is a consequence of \ref{decf(M)} via ordinary computations.
\end{proof}

\section{Ordered quadratically closed fields and real closed fields.}

\begin{defi}\label{defQCO}
Assume that $\mathbb{K}$ is an ordered field. We say that $\mathbb{K}$ is a \emph{ordered quadratically closed field}, if every positive element of $\mathbb{K}$ has a square root in $\mathbb{K}$. 

For every $a \in \mathbb{K}$, $a>0$ we denote by $\sqrt{a}$ the unique positive square root of $a$ in $\mathbb{K}$.

For this notion we follow \cite{Lang2002} Ch.\,IX p.462 rather than other definitions of quadratically closed field in literature.
\end{defi}

\begin{rem}
It is known that an ordered quadratically closed field has a unique order (as field) and it has characteristic $0$ (see for instance \cite{Raj1993} Thm.\,15.3). Hence, from \ref{car0sommadiretta}, we have the canonical decomposition $\overline{\mathbb{K}} = \mathbb{K} \oplus Ker(H_\mathbb{K})$ as $\mathbb{K}$-vector space.
\end{rem}

\begin{defi}\label{defRC}
The field $\mathbb{K}$ is said to be a \emph{real closed field}, if it can be endowed with a structure of ordered field such that its positive elements have a square root in $\mathbb{K}$ and any polynomial of odd degree of $\mathbb{K}[X]$ has a root in $\mathbb{K}$.
\end{defi}

\begin{rem}
It follows directly from the definitions that every real closed field is an ordered quadratically closed field.

It is known that, for every ordered field $\mathbb{K}$, there exists an algebraic extension, contained in $\overline{\mathbb{K}}$, which is real closed and whose order extends the order of $\mathbb{K}$ and, moreover, that any two such extensions are $\mathbb{K}$-isomorphic (see for instance \cite{Jac1989} Thm.\,11.4 or \cite{Raj1993} Thm.\,15.9).

We call any such extension $\mathbb{L}$ a \emph{real closure of the ordered field} $\mathbb{K}$ \emph{in} $\overline{\mathbb{K}}$.

Note that $\overline{\mathbb{K}}$ is the algebraic closure of $\mathbb{L}$ too.
\end{rem}

\begin{rem}\label{real-closed}
For more information, further characterizations and properties of real closed fields we refer for instance to \cite{Lang2002} Ch.XI \S2 and to \cite{Raj1993} Ch.\,15.

In particular it is known that $\mathbb{K}$ is a real closed field if and only if  $\sqrt{-1} \notin \mathbb{K}$ and $\mathbb{K}(\sqrt{-1})$ is algebraically closed (see for instance \cite{Raj1993} p.\,221 (i)).

In \ref{real-closed-equiv} we point out other simple characterizations of real closed fields, useful in our setting.
\end{rem}

\begin{prop}\label{parti-Re-Im}
Assume that $\mathbb{K}$ is an ordered quadratically closed field, choose one of its real closures, $\mathbb{L}$, in $\overline{\mathbb{K}}$ and a square root $\sqrt{-1} \in \overline{\mathbb{K}}$ of $-1$. 

a) For every element $z \in \overline{\mathbb{K}}$ there exist, uniquely determined by $\mathbb{L}$ and $\sqrt{-1}$, elements $x, y \in \mathbb{L}$ such that $z = x + \sqrt{-1}\,y$. 

We denote $x=\mathbf{Re}(z)$ and $y = \mathbf{Im}(z)$: the \emph{real} and the \emph{imaginary} part of $z$ (depending on $\mathbb{L}$ and $\sqrt{-1}$).

b) For every element $z \in \overline{\mathbb{K}}$ of degree $2$ over $\mathbb{K}$, $\mathbf{Re}(z)$ and $\mathbf{Im}(z)$ are both elements of $\mathbb{K}$ and moreover $H_\mathbb{K} (z) = \mathbf{Re}(z)$ and $V_\mathbb{K} (z) = \sqrt{-1}\,\mathbf{Im}(z)$; hence, in this case, $\mathbf{Re}(z)$ and $\mathbf{Im}(z)$ are independent of $\mathbb{L}$.
\end{prop}

\begin{proof}
Part (a) follows from \ref{real-closed} since $\overline{\mathbb{K}} = \mathbb{L}(\sqrt{-1})$.

Let $z \in \overline{\mathbb{K}}$ as in (b).
From \ref{decomp-traccia}, we can write $z= H_\mathbb{K} (z) + V_\mathbb{K} (z)$ with $V_\mathbb{K} (z)$ root of the polynomial $X^2 - V_{\mathbb{K}}(z)^2$, which is irreducible over $\mathbb{K}$, since $z$ and $V_{\mathbb{K}}(z)$ have degree $2$ over $\mathbb{K}$. Hence  $- V_{\mathbb{K}}(z)^2 >0$ and so $\pm \sqrt{- V_{\mathbb{K}}(z)^2}$ are both elements of the ordered quadratically closed field $\mathbb{K}$. Now $V_\mathbb{K} (z) = \sqrt{-1} \,[\pm \sqrt{- V_{\mathbb{K}}(z)^2}]$ and we conclude (b) by uniqueness of the decomposition in (a).
\end{proof}

\begin{lemma}
Let $(\mathbb{K}, |.|)$ be a valued field  and $\lambda \in \overline{\mathbb{K}} \subseteq \overline{\mathbb{K}^c}$.
Then $(\mathbb{K}(\lambda))^c = \mathbb{K}^c(\lambda)$.
\end{lemma}
\begin{proof}
The element $\lambda$ is algebraic over $\mathbb{K}^c$ too. Hence, by \cite{Lang2002} Ch.\,XII Prop. 2.5, we get that $\mathbb{K}^c(\lambda)$ is complete. Since it contains $\mathbb{K}(\lambda)$, it contains also its completions and this gives an inclusion.

Now let $\vartheta \in \mathbb{K}^c(\lambda)$ be of the degree $l$ over $\mathbb{K}^c$. Therefore $\vartheta = \sum_{i=0}^{l-1} h_i \lambda^i$ with $h_0, \cdots , h_{l-1} \in \mathbb{K}^c$. Since $\mathbb{K}$ is dense in $\mathbb{K}^c$, there exist sequences in $\mathbb{K}$, $\{k_m^{(i)}\}_{m \in \mathbb{N}}$, $0 \le i \le l-1$, such that $k_m^{(i)}$ converges to $h_i$.
Since the topology over $\mathbb{K}^c(\lambda)$ is the product topology (see \cite{Lang2002} Ch.\,XII, Prop.\,2.2), we have that 

$k_m^{(0)} + k_m^{(1)} \lambda + \cdots + k_m^{(l-1)} \lambda^{l-1}$ is a sequence in $\mathbb{K}(\lambda)$ which converges to $\vartheta$. So $\vartheta \in (\mathbb{K}(\lambda))^c$.
\end{proof}

\begin{prop}\label{completamento-RCVal}
Let $(\mathbb{K}, |.|)$ be a real closed valued field and denote by $\mathbb{K}^c$ its completion.

If $\sqrt{-1} \in \mathbb{K}^c$, then $\mathbb{K}^c$ is algebraically closed.

If $\sqrt{-1} \notin \mathbb{K}^c$, the $\mathbb{K}^c$ is real closed.
\end{prop}

\begin{proof}
By the previous Lemma we have: $\overline{\mathbb{K}}^c = \mathbb{K}(\sqrt{-1})^c = \mathbb{K}^c (\sqrt{-1})$. 

Since the completion of an algebraically closed field is algebraically closed too (remember \ref{completamento} (a)), $\mathbb{K}^c (\sqrt{-1})$ is algebraically closed.

Hence if $\sqrt{-1} \in \mathbb{K}^c$, then $\mathbb{K}^c$ is algebraically closed. Otherwise $\mathbb{K}^c$ is real closed by the characterization recalled in \ref{real-closed}. 
\end{proof}

\begin{cor}\label{cor-dopo-completamento}
If $(\mathbb{K}, |.|)$ be a real closed valued field of p-adic type for some prime $p$ (remember \ref{exp-cos'}), then its completion $\mathbb{K}^c$ is algebraically closed.
\end{cor}
\begin{proof}
Let us consider the sequence $\{x_n\}_{n \ge 1}$, $x_n=\sqrt{p^n -1}$. Since $\mathbb{K}$ is real closed (hence ordered and quadratically closed), $x_n \in \mathbb{K}$. Now $x_n^2 +1 = p^n \to 0$ in $(\mathbb{K}, |.|)$.

If there exists a subsequence $\{x_{n_k}\}$ such that $(x_{n_k} +\sqrt{-1}) \to 0$, then $\pm \sqrt{-1} \in \mathbb{K}^c$.

Otherwise there exists a real number $\delta >0$ such that $|x_n + \sqrt{-1}| \ge \delta >0$ for every $n$. In this case $|x_n - \sqrt{-1}| = \dfrac{|x_n^2 +1|}{|x_n + \sqrt{-1}|} \le \dfrac{|x_n^2 +1|}{\delta} \to 0$. So $x_n \to \sqrt{-1}$ and again $\sqrt{-1} \in  \mathbb{K}^c$. Hence, by the previous Proposition, $\mathbb{K}^c$ is algebraically closed.
\end{proof}

\section{Matrices on ordered quadratically closed fields.}

\begin{remdef}\label{norm-Frob}
Let $\mathbb{K}$ be an ordered quadratically closed field and $M \in M_n(\mathbb{K}) \setminus \{0\}$ be a semisimple matrix of splitting bound at most $2$ over $\mathbb{K}$. As remarked in \ref{parti-Re-Im}, after choosing $\sqrt{-1} \in \overline{\mathbb{K}}$, the decomposition 

$\lambda_i = \mathbf{Re}(\lambda_i) + \sqrt{-1}\,\mathbf{Im}(\lambda_i)$ is well-defined, being independent on the choice of a real closure of $\mathbb{K}$ into $\overline{\mathbb{K}}$.

Remembering the definitions of the matrices $\mathcal{A}_i(M)$ and $\mathcal{B}_{j}(M)$ in \ref{dec-fine-unica}, we denote:

$\mathbf{A}_i(M) = \mathcal{A}_i(M)$ for every $i= 1, \cdots , s$, \  and

$\mathbf{B}_j(M) = \dfrac{\mathcal{B}_j(M)}{\sqrt{\mathbf{Im}(\lambda_j)^2}}$ for every $j =1, \cdots , t$. 

Since $\mathbb{K}$ is an ordered quadratically closed field, the matrices $\mathbf{A}_i(M)$, $\mathbf{B}_j(M)$ have coefficients in $M_n(\mathbb{K})$ and are polynomial expressions of $M$. Moreover

$\mathbf{A}_i(M) \mathbf{A}_j(M) = \delta_{ij} \mathbf{A}_i(M)$ for every $i,j$; 

$\mathbf{A}_i(M) \mathbf{B}_j(M) = \mathbf{B}_j(M) \mathbf{A}_i(M) = 0$ for every $i,j$;

$\mathbf{B}_i(M) \mathbf{B}_j(M) =0$ for every $i \ne j$ and

$(\mathbf{B}_j(M))^3 = - \mathbf{B}_j(M)$ for every $j$.

Then by \ref{dec-fine-unica} we get
\begin{center}
$(*)  \ \ \ \ M =  \sum_{i=1}^s \gamma_i \, \mathbf{A}_i(M) - \sum_{j=1}^t \mathbf{Re}(\lambda_j) \, \mathbf{B}_j(M)^2 + \sum_{j=1}^t \sqrt{\mathbf{Im}(\lambda_j)^2} \, \mathbf{B}_j(M).
$
\end{center}
We call the above decomposition \emph{normalized fine Frobenius decomposition of} $M$ and the matrices $\mathbf{A}_i(M)$'s and $\mathbf{B}_j(M)$'s its \emph{normalized fine Frobenius covariants}.

Using \ref{dec-fine-unica} it is possible to prove that a normalized fine Frobenius decomposition exists only for semisimple nonzero matrices with splitting bound at most $2$ over the ordered quadratically closed field $\mathbb{K}$ and that such decomposition is uniquely determined up the order of the addends; analogously the normalized fine Frobenius covariants are uniquely determined from their properties, stated above, and from the previous decomposition.

As in \ref{oss-dopo-dec-fine-unica} (b) we can retrieve the possible eingenvalue $0$ by posing $\gamma_0=0$ and $\mathbf{A}_0(M) = I_n -\sum_{i=1}^s \mathbf{A}_i(M) + \sum_{j=1}^t \mathbf{B}_j(M)^2$ which allows to write the normalized fine Frobenius decomposition of $M$ as:
\begin{center}
$(*') \ \ \ \ M =  \sum_{i=0}^s \gamma_i \, \mathbf{A}_i(M) - \sum_{j=1}^t \mathbf{Re}(\lambda_j) \, \mathbf{B}_j(M)^2 + \sum_{j=1}^t \sqrt{\mathbf{Im}(\lambda_j)^2} \, \mathbf{B}_j(M).
$
\end{center}
\end{remdef}

\begin{prop}\label{**}
Assume that $(\mathbb{K}, |\cdot |)$ is an ordered quadratically closed valued field, that $f(X)$ is a series with coefficients in $\mathbb{K}$ and that  
$M \in \widehat{\Omega}_{f, \mathbb{K}} \setminus \{0\}$ be a semisimple matrix with splitting bound  at most $2$ over $\mathbb{K}$ and with normalized fine Frobenius decomposition:

$M =  \sum_{i=0}^s \gamma_i \, \mathbf{A}_i(M) - \sum_{j=1}^t \mathbf{Re}(\lambda_j) \, \mathbf{B}_j(M)^2 + \sum_{j=1}^t \mathbf{Im}(\lambda_j) \, \mathbf{B}_j(M)$,

where the $\lambda_j$'s are the eigenvalues of $M$ not in $\mathbb{K}$, having positive imaginary part (so $\mathbf{Im}(\lambda_j)= \sqrt{\mathbf{Im}(\lambda_j)^2}$).

Then
\begin{center}
$(**) \ \ \ 
f(M) = \sum_{i=0}^s f(\gamma_i) \mathbf{A}_i(M) - \sum_{j=1}^t  \mathbf{Re}f(\lambda_j) \, \mathbf{B}_j(M)^2  + \sum_{j=1}^t \mathbf{Im}f(\lambda_j) \mathbf{B}_j(M),
$
\end{center}

with $f(\gamma_i), \ \mathbf{Re}f(\lambda_j) , \ \mathbf{Im}f(\lambda_j) \ \in \mathbb{K}^c$ for every $i,j$.
\end{prop}
\begin{proof}
It follows directly from  \ref{decf(M)}.
\end{proof}

\begin{examples}\label{exp-cos}
We go back to the examples in \ref{exp-cos'} under the same conditions over $\mathbb{K}$ of the previous Proposition.

If $M \in \widehat{\Omega}_{\mathbb{K}} \setminus \{0\}$ is semisimple with splitting bound at most $2$ over $\mathbb{K}$, then 

$\exp(M)=\sum_{i=0}^s \exp(\gamma_i) \mathbf{A}_i(M) - \sum_{j=1}^t \exp(\mathbf{Re}(\lambda_j)) \cos(\mathbf{Im}(\lambda_j) ) \mathbf{B}_j(M)^2 +$

$\ \ \ \ \ \ \ \ \ \ \ \ \ \ \ \ \ \ \ \ \ \ \ \ \ \ \ \ \ \ \ \ \ \ \ \ \ \ \ \ \  +\sum_{j=1}^t \exp(\mathbf{Re}(\lambda_j)) \sin(\mathbf{Im}(\lambda_j)) \mathbf{B}_j(M)$

and

$\cos(M)= \sum_{i=0}^s \cos(\gamma_i) \mathbf{A}_i(M) - \sum_{j=1}^t \cos(\mathbf{Re}(\lambda_j)) \cosh(\mathbf{Im}(\lambda_j)) \mathbf{B}_j(M)^2 +$

$\ \ \ \ \ \ \ \ \ \ \ \ \ \ \ \ \ \ \ \ \ \ \ \ \ \ \ \ \ \ \ \ \ \ \ \ \ \ \ \ \  - \sum_{j=1}^t \sin(\mathbf{Re}(\lambda_j)) \sinh(\mathbf{Im}(\lambda_j)) \mathbf{B}_j(M),$

where the $\lambda_j$'s are the eigenvalues of $M$ not in $\mathbb{K}$, having positive imaginary part.

The previous formula of $\exp(M)$ extends the classical Rodrigues' formula for the exponential of a real skew symmetric matrix (see for instance \cite{GaXu2002} Thm.2.2).
\end{examples}

\begin{prop}\label{real-closed-equiv} The following assertions are equivalent:

a) $\mathbb{K}$ is real closed;

b) $\mathbb{K}$ is not algebraically closed, $char(\mathbb{K}) \ne 2$ and the irreducible polynomials of $\mathbb{K}[X]$ have degree at most $2$;

c) $\mathbb{K}$ is a perfect field of characteristic different from $2$ and the group $Aut(\overline{\mathbb{K}}/\mathbb{K})$ has order $2$.

d) the $\mathbb{K}$-involution of $\overline{\mathbb{K}}$ is an element of $Aut(\overline{\mathbb{K}}/\mathbb{K})$ different from the identity;

e) $\mathcal{R}_\mathbb{K} = \overline{\mathbb{K}}$ and $Ker(H_\mathbb{K})$ is the $\mathbb{K}$-vector space generated by $\sqrt{-1}$;

f) $\mathbb{K}$ is not algebraically closed, $char(\mathbb{K}) \ne 2$ and every matrix $M \in M_n(\mathbb{K})$ has splitting bound al most $2$  over $\mathbb{K}$.
\end{prop}

\begin{proof}
For the equivalence between (a) and (b), first we remark that one implication follows from \ref{real-closed}.

For the converse it suffices to prove that $\overline{\mathbb{K}} = \mathbb{K}(\sqrt{-1})$, since $\mathbb{K}$ is not algebraically closed. Let $t \in \overline{\mathbb{K}} \setminus \mathbb{K}$, so it has degree $2$ and $t \in \mathcal{R}_\mathbb{K}$, since $char(\mathbb{K}) \ne 2$. By \ref{decomp-traccia}, we decompose  $t = \alpha + \beta$ as sum of an element $\alpha \in \mathbb{K}$ and of an element $\beta \in Ker(H_\mathbb{K}) \setminus \{0\}$. By \ref{car-non-2} (a) the conjugated of $t$ is it $\mathbb{K}$-involution $\overline{t} = \alpha - \beta$. This implies that the minimal polynomial of $t$ over $\mathbb{K}$ is $p(X) = X^2 - 2 \alpha X + \alpha^2 - \beta^2$, therefore the reduced form of $p(X)$ is $\tilde{p}(X) = X^2 - \beta^2$. Hence $\beta^2 \in \mathbb{K}$ while $\beta \notin \mathbb{K}$.

Now we consider the polynomial of $\mathbb{K}[X]$

$q(X)= X^4 - \beta^2 = (X - \sqrt{\beta}) (X + \sqrt{\beta}) (X - \sqrt{- \beta}) (X + \sqrt{-\beta})$ 

with its factorization in $\overline{\mathbb{K}}[X]$ (note that its roots are not in $\mathbb{K}$). Since $q(X)$ has degree $4$, it is reducible over $\mathbb{K}$, so it is product of two irreducible polynomials of $\mathbb{K}[X]$. Since $\beta \notin \mathbb{K}$, one of the two factors must have the form $(X- \sqrt{\beta}) (X \pm \sqrt{-\beta})$ and therefore $\sqrt{-\beta^2} \in \mathbb{K} \setminus \{0\}$. Hence $\beta = \pm \sqrt{-\beta^2} \sqrt{-1} \in \mathbb{K}(\sqrt{-1})$. This implies that $t = \alpha + \beta \in \mathbb{K}(\sqrt{-1})$, therefore $\overline{\mathbb{K}} \setminus \mathbb{K} \subseteq \mathbb{K}(\sqrt{-1})$ and so $\overline{\mathbb{K}} = \mathbb{K}(\sqrt{-1})$.

Assume the condition (c). $Aut(\overline{\mathbb{K}}/\mathbb{K})$ acts transitively on every conjugacy class over $\mathbb{K}$, so the irreducible polynomials in $\mathbb{K}[X]$ have degree at most $2$, because $\mathbb{K}$ is perfect. This gives that (c) implies (b).
On the other hand it is obvious that (a) implies (c).

Note that the assumptions in (d) imply that $\mathbb{K}$ is not algebraically closed, $char(\mathbb{K}) \ne 2$ and $\mathcal{R}_\mathbb{K} = \overline{\mathbb{K}}$. 

It is trivial that (a) implies (d).

Now assume (d). Let $\lambda = \alpha + \beta \in \overline{\mathbb{K}} = \mathcal{R}_\mathbb{K}$ with its $\mathbb{K}$-decomposition. In particular $\mathbb{K}(\lambda) = \mathbb{K}(\beta)$. From (d) we have: $\overline{\beta^2} = \overline{\beta}^2 = (-\beta)^2 = \beta^2$. Hence, by \ref{involuzione}, we get that $\beta^2 \in \mathbb{K}$ and so both $\beta$ and $\lambda$ have degree $2$ over $\mathbb{K}$. This gives (b).

Next we prove the equivalence between (a) and (e). Assume first (e). By \ref{car0sommadiretta},  $\mathcal{R}_\mathbb{K}  = \overline{\mathbb{K}} = \mathbb{K} \oplus Ker(H_\mathbb{K})$, so $\overline{\mathbb{K}}$ is the $\mathbb{K}$-vector space of dimension $2$ generated by $1$ and $\sqrt{-1}$, hence $\sqrt{-1} \notin \mathbb{K}$ and $\mathbb{K}(\sqrt{-1}) = \overline{\mathbb{K}}$ is algebraically closed.

For the converse, we have $char(\mathbb{K}) = 0$ (so $\mathcal{R}_\mathbb{K} = \overline{\mathbb{K}}$) and every element in $\overline{\mathbb{K}} \setminus \mathbb{K}$ is algebraic of order $2$ over $\mathbb{K}$. By \ref{car-non-2} (e), $\beta \in \overline{\mathbb{K}}$ belongs to $Ker(H_\mathbb{K})$ if and only if $\beta = 0$ or $\beta \notin \mathbb{K}$ and $\beta = \pm \sqrt{t}$ with $t \in \mathbb{K}$. In this last case $t$ is negative since $\mathbb{K}$ is real closed, hence $\beta = \pm \sqrt{-t} \, \sqrt{-1}$ with $\sqrt{-t} \in \mathbb{K} \setminus \{0\}$. So $\beta \in Ker(H_\mathbb{K})$ if and only if $\beta = k \sqrt{-1}$ with $k \in \mathbb{K}$.

Now (b) implies (f) by obvious reasons. For the converse it suffices to remember that every monic polynomial of degree $n$ in $\mathbb{K}[X]$ is the minimal polynomial of its \emph{companion matrix} which belongs to $M_n(\mathbb{K})$.
\end{proof}

\begin{rem}\label{caso-reale}
Assume that $\mathbb{K}$ is real closed and, as for an ordered quadratically closed field, choose one of the two roots of $X^2 +1$, denoted by $\sqrt{-1}$. 
For every $z \in \overline{\mathbb{K}}$ there are uniquely determined $a, b \in \mathbb{K}$ such that $z = a + \sqrt{-1} \, b$, as above denoted $a=\mathbf{Re}(z)$ and $b = \mathbf{Im}(z)$ (the \emph{real} and the \emph{imaginary} part of $z$).

Note that, in this case, the generator of $Aut(\overline{\mathbb{K}}/\mathbb{K})$ is necessarily the $\mathbb{K}$-involution.

We draw attention on point (f) of \ref{real-closed-equiv}: every $M \in  M_n(\mathbb{K})$ has splitting bound at most $2$ over $\mathbb{K}$. 
Therefore, if $\mathbb{K}$ is real closed, then every semisimple $M \in  M_n(\mathbb{K})\setminus \{0\}$ has a unique normalized fine Frobenius decomposition as in \ref{norm-Frob} $(*)$ and $(*')$.

Analogously if $(\mathbb{K}, |\cdot|)$ is a real closed valued field and $f(X)$ is a series with coefficients in $\mathbb{K}$, then for every semisimple matrix $M \in \widehat{\Omega}_{f, \mathbb{K}} \setminus \{0\}$, we have the formula $(**)$ as in \ref{**}. Similar formulas, as in \ref{exp-cos}, hold in case of $\exp(M)$ and $\cos(M)$, for every semisimple matrix
$M \in \widehat{\Omega}_{\mathbb{K}} \setminus \{0\}$. 
\end{rem}


\begin{thebibliography}{}

\bibitem[Dolcetti-Pertici 2017]{DoPe2017} DOLCETTI Alberto, PERTICI Donato, ``Some remarks on the Jordan-Chevalley decomposition'', https://arxiv.org/abs/1707.01794.

\bibitem[Gallier-Xu 2002]{GaXu2002} GALLIER Jean, XU Dianna, ``Computing exponential of skew-symmetric matrices and logarithms of orthogonal matrices'', \emph{International Journal of Robotics and Automation}, Vol. 17, No. 4, 10--20.

\bibitem[Garver 1927/28]{Gar1927} GARVER Raymond, ``The Tschirnhaus Transformation'', \emph{Ann. of Math.} (2)  29,  no. 1-4, 319--333. 

\bibitem[Horn-Johnson 1991]{HoJ1991} HORN Roger A.,  JOHNSON Charles R., \emph{Topics in matrix analysis}, Cambridge University Press, Cambridge.

\bibitem[Hungerford 1974]{Hun1974} HUNGERFORD Thomas W., \emph{Algebra}, GTM 73, Springer-Verlag, New York.

\bibitem[Jacobson 1989]{Jac1989} Jacobson Nathan, \emph{Basic Alcebra II}, Second Edition, H. W. Freeman and Company, New York. 

\bibitem[Lang 2002]{Lang2002} LANG Serge, \emph{Algebra}, GTM 211, Revised Third Edition, Springer-Verlag, New York.

\bibitem[Lorenz 2008]{Lor2008} LORENZ Falko, \emph{Algebra. Volume II: Fields with Structure, Algebras and Advanced Topics}, Universitext, Springer, New York

\bibitem[Rajwade 1993]{Raj1993} RAJWADE A. R., \emph{Squares}, London Mathematical Society Lecture Note Series 171, Cambridge University Press, Cambridge.

\bibitem[Schikhof 1984]{Schik1984} SCHIKHOF W. H., \emph{Ultrametric calculus. An introduction to p-adic analysis}, Cambridge Studies in advanced mathematics 4, Cambridge University Press, Cambridge.

\bibitem[Tignol 2001]{Tignol2001} TIGNOL Jean-Pierre, \emph{Galois' Theory of Algebraic Equations}, World Scienfific, Singapore.

\bibitem[Warner 1989]{War1989} WARNER Seth, \emph{Topological fields}, Mathematics Studies 157, North Holland, Amsterdam.

\bibitem[Yanai-Takeuchi-Takane 2011]{YTT2011} YANAI Haruo, TAKEUCHI Kei, TAKANE Yoshio, \emph{Projection Matrices, Generalized Inverse Matrices, and Singular Value Decomposition}, Springer, New York.
\end{thebibliography}
\end{document}